\documentclass[fleqn,reqno,11pt,a4paper,final]{amsart}

\usepackage[a4paper,left=30mm,right=30mm,top=30mm,bottom=30mm,marginpar=20mm]{geometry}
\usepackage[utf8]{inputenc}
\usepackage{amsmath}            
\usepackage{amsfonts}           
\usepackage{amsthm}             
\usepackage{amssymb}            
\usepackage{mathrsfs}           

\usepackage{thm_en}      
\usepackage{mymathfonts} 

\usepackage{accents}            

\title[weak solutions of conservation laws and energy/entropy conservation]{A note on weak solutions of conservation laws and energy/entropy conservation}

\author{Piotr Gwiazda$^1$ \and Martin Mich\'alek$^2$ \and Agnieszka \'{S}wierczewska-Gwiazda$^3$}
\begin{document}
    \maketitle
    
    \let\thefootnote\relax\footnote{{\hspace{-0.4cm}1.  Institute of  Mechanics, Polish Academy of Sciences, \'Sniadeckich 8, 00-656 Warszawa, Poland, email:  pgwiazda@mimuw.edu.pl\\
2. Institute of Mathematics of the Academy of Sciences of the Czech Republic, \v{Z}itn\'{a} 25, CZ-115 67 Praha~1, Czech Republic, email:  michalek@math.cas.cz\\
3. Institute of Applied Mathematics and Mechanics, University of Warsaw, Banacha 2, 02-097 Warszawa, Poland, phone: {\it +48 22 5544115}, email: aswiercz@mimuw.edu.pl\\
}}

\begin{abstract}
    A common feature of systems of conservation laws of continuum physics is that they are endowed with natural companion laws which are in such case most often related to the second law of thermodynamics.     This observation easily generalizes to any symmetrizable system of conservation laws. They are endowed with nontrivial companion conservation laws, which are immediately satisfied by classical solutions. Not surprisingly, weak solutions may fail to satisfy companion laws, which are then often relaxed from equality to inequality and overtake a role of a physical admissibility condition for weak solutions. 
    
    We want to answer the question what is a critical regularity of  weak solutions to a general system of conservation laws to satisfy an associated  companion law as an equality. An archetypal 
    example of such result was derived for the incompressible Euler system by Constantin et al. (\cite{ConstETiti}) in the context of the seminal Onsager's conjecture. 
    
    This general result can serve as a simple criterion to numerous systems of mathematical physics to prescribe the regularity of solutions needed for an appropriate companion law to be satisfied. 
    
    \noindent{\bf Keywords}: energy conservation, first order hyperbolic system, Onsager's conjecture

\end{abstract}

\section{Introduction}

The passing decade has been to a significant extend directed to solving the famous conjecture of Onsager saying that solutions to incompressible Euler system conserve total kinetic energy as long as they are  H\"older continuous  with a 
H\"older exponent $\alpha>1/3$. Otherwise they may dissipate the energy. 

The ideas used to prove the celebrated Nash-Kuiper theorem appeared to have wide applicability in the context of fluid mechanics, and incompressible Euler system in particular. Interestingly, the construction of weak solutions via appropriate refinement of the method of convex integration allowed to generate solutions with a regularity as exactly prescribed by Onsager that do not coserve the energy. 
 We shall summarize in a sequel the recent achievements in this direction, however our main interest in the current paper is aimed at an analogue and generalization of the first part of the Onsager's statement. 
 The positive direction of this claim was fully solved by Constantin et al. already in the early nineties, cf.~\cite{ConstETiti}, see also~\cite{cheskidov, DuRo2000, eying}.
 A sufficient regularity for the energy to be conserved has been established  for a variety of models, including the incompressible inhomogeneous Euler system and the compressible Euler in \cite{Feireisletal}, the incompressible inhomogeneous Navier-Stokes system in \cite{LeSh2016}, compressible Navier-Stokes in~\cite{Yu}
and equations of magneto-hydrodynamics in \cite{Cafetal}. 

The above list gives a flavor of how broad is the class of systems for which one can specify the regularity of weak solutions  which provides the energy to be conserved. This motivates us, instead of developing tools for another dozens of systems, to look at general systems of conservation laws. Apparently one can prescribe the condition for weak solutions providing that in addition to a conservation law they will satisfy a companion conservation law. To make the statement more precise, let us consider
%
%
%
a conservation law, not necessarily hyperbolic, in a general form
\begin{equation}
  \label{eq:balance}
  \diver_{X}(G(U(X))=0 \quad \quad \mbox{for } X\in \CX
\end{equation}
for an unknown (vector) function  $U=U(X)\colon \CX\to \CO $  and a given matrix field $G\colon \CO \to \BM^{n\times (k+1)}$.  Let us assume that $\CO$ and $\CX$ are open sets, $\CX \subseteq \BR^{k+1}$ or $\CX \subseteq \BR\times\BT^{k}$ and
$\CO\subseteq \BR^n$, where $\BT^{k}$ denotes the  flat torus of dimension $k$ (imposing the periodic boundary conditions). We denote  $X = (x_0,x_1,\dots, x_k)^T$ the standard coordinates on $\BR^{k+1}$ or $\BR\times \BT^k$ and we consider on $\CO$ the coordinates $Y=(y_1,\dots,y_n)^T$ with respect to the canonical basis. 
For a matrix field $M=(M_{i,j})_{i=1,\dots, n,\ j=0,\dots, k}$, $M_{i,j}\colon \BR^n \to \BR$, we 
denote $M_{j}$ the $j$--th  column vector. 
Moreover, we use the standard definition
\begin{equation*}
  \diver_{X} M(X) = 
  \sum_{j=0}^k \partial_{x_j}M_{j}(X).
\end{equation*}
We denote by $D_X$ (respectively $D_Y$, $D_U$) the differential 
($D_X=(\partial_{x_0},\dots,\partial_{x_k})$) with respect to variables $X$ (respectively $Y$, $U$). 

Following the notation in~\cite{dafermos} we shall say that a smooth function $Q\colon \CO \to \BR^{s\times (k+1)}$ is a {\it companion} of $G$ if there exists a smooth function ${\mathcal B}\colon \CO \to \BM^{s\times n}$ such that 
\begin{equation}           
 \label{eq:algebraic_cons}
  D_{U}Q_j(U)={\mathcal B}(U)  D_{U}G_j(U)\quad \mbox{for all $U\in \CO$, $j\in\{0,\dots,k\}$}.
\end{equation}
Observe that for any classical solution $U$ of \eqref{eq:balance}, we obtain 
\begin{equation}
    \label{eq:companion}
  \diver_X (Q(U(X)))=0 \quad \mbox{for $X\in \CX$}
\end{equation}
where by a {\it classical solution} we mean a Lipschitz continuous vector field $U$ satisfying \eqref{eq:balance} for almost all $X\in\CX$. 
Identity \eqref{eq:companion} is called a {\it companion law} associated to $G$ (see e.g. \cite{dafermos}). In many applications, which we partially recall in  Section \ref{sec:appl}, some relevant companion laws are {\it conservation of energy} or {\it conservation of entropy}.  
Before we discuss the relations between weak solutions and companion laws, let us remark that it was observed by Godunov \cite{godunov} that  systems of conservation laws are symmetrizable if and only if they are endowed with nontrivial companion laws.

We consider the standard definition of weak solutions to a conservation law 
\begin{dfn}
    We call the function $U\colon \CX \to \CO$ a weak solution to \eqref{eq:balance} if
    $G(U)$ is locally integrable in $\CX$ and the equality
\begin{equation}    
    \label{eq:weak_sol}
  \int_{\CX} G(U(X))\colon D_X \psi(X) \de{X}=
  0
\end{equation}
holds for all smooth test functions $\psi\colon \CX \to \BR^n$ with a compact support in $\CX$.
\end{dfn}
Analogously, we can define weak solutions to \eqref{eq:companion}, however weak solutions of \eqref{eq:balance} may not necessarily  be weak solutions  also to \eqref{eq:companion}. 
The main question we deal with in this paper reads as follows:
\emph{What are  sufficient conditions for a weak solution of \eqref{eq:balance} to satisfy also \eqref{eq:companion}?} 

Let us comment in more detail results related to the question of energy conservation for weak solutions of some conservation laws. 
Both parts of Onsager's conjecture for the incompressible inviscid Euler system have been resolved. Due to recent results of Isett \cite{isett}  and Buckmaster et al.\,\cite{delell_rec} we know  there exist solutions of the incompressible Euler equations of class $C([0,T];C^{1\slash 3}(\BT^3))$ which do not satisfy the energy equality. These results were preceded with a series of papers showing firstly existence
of bounded (\cite{DLS09}), later continuous (\cite{DLS13}) and H\"older continuous (\cite{DLS14}) solutions with $\alpha=1/10$. The further results aimed to increasing the H\"older exponent, see \cite{Buc15, BSLISJ15, BDKS16,  Ise13}.

In the context of our studies, the second part of Onsager's conjecture is more relevant. Constantin et al. \cite{ConstETiti} showed the conservation of the global kinetic energy  if the velocity field $u$ is of the class $L^3(0,T;B^{\alpha}_{3,\infty}(\BT^3))\cap C([0,T];L^2(\BT^3))$ whenever $\alpha>\frac 13$, see also \cite{eying}. Here 
$B^{\alpha}_{3,\infty}$ stands for a Besov space (definition is recalled in Section~\ref{sec:besov}).  For the same system it was observed by Cheskidov et al.\,in \cite{cheskidov} that it is sufficient 
for $u$ to belong to a larger space $L^3(0,T;B^{1\slash 3}_{3,q}(\BT^3))$ where $q\in(1,\infty)$. We refer the reader to \cite{shvydkoy_energy} and \cite{shvydkoy} 
for more refinements in the case of the incompressible Euler system.
For the incompressible inviscid equations of magneto--hydrodynamics, a result comparable to \cite{ConstETiti} was proved by Caflisch et al.\,\cite{Cafetal}, see also \cite{KangLee}.

The standard technique developed in \cite{ConstETiti} is based on the convolution of the Euler system with a standard family of mollifiers. The crucial part of the proof is 
then to estimate an appropriate nonlinear commutator. The most of the mentioned results have been derived for systems with bilinear non--linearity.

Recently, similar results for the compressible Euler system were presented by Feireisl et al. in \cite{Feireisletal}. A sufficient condition for the energy conservation is that the solution belongs to $B^{\alpha}_{3, \infty}((0,T)\times \BT^3)$ with $\alpha>1\slash 3$. Up to our knowledge, this was the first result treating nonlinearity which is not in a multilinear form.  
We extend this approach to a general class of conservation laws of the form~\eqref{eq:balance}. Let us mention that we are not aware of any reference where the problem would be treated in such generality. We believe that this general scenario might be of interest.
Moreover, at least the application on the equations of polyconvex elastodynamics (Subsection \ref{subsec:polyconv}) is an original contribution of this paper. 

Let us present the main results of the paper. For the notation, we refer the reader to Section \ref{sec:besov}. 
\begin{thm}
    \label{thm:bes}
    Let $U\in B^{\alpha}_{3,\infty}(\CX; \CO)$ be a weak solution of \eqref{eq:balance} with $\alpha>\frac{1}{3}$. 
    Assume that $G \in C^2(\CO;\BM^{n\times (k+1)})$ is endowed with a companion law with flux 
    $Q\in C(\CO;\BM^{1\times (k+1)})$ for which there exists  ${\mathcal B}\in C^1(\CO;\BM^{1\times n})$ related through  identity \eqref{eq:algebraic_cons}
    and all the following conditions hold
    \begin{equation}
        \label{eq:assumpt_convex}
         \left.
         \begin{aligned} 
         \mbox{$\CO$ is convex},
         \\
         {\mathcal B}\in W^{1,\infty}(\CO;\BM^{1\times n}),
         \\
         |Q(V)|\leq C(1+|V|^3)\   \mbox{for all $V\in\CO$},
         \\
         \sup_{i,j \in{1,\dots,d}}\|\partial_{U_i}\partial  _{U_j} G(U)\|_{C(\CO;\,\BM^{n\times (k+1)})}<+\infty.
         \end{aligned}
         \right\}
    \end{equation}    
    Then $U$ is a weak solution of the companion law \eqref{eq:companion}  with the flux $Q$.
\end{thm}

\begin{rem}
    \begin{itemize}
      \item We consider only a special case when the companion law is a scalar equation. If $Q\colon \CO \to \BM^{s\times (k+1)}$ and $s>1$, we can apply Theorem \ref{thm:bes} to each row of~\eqref{eq:companion}. 
      \item  The growth condition of $Q$ can be relaxed whenever $B^{\alpha}_{3,\infty}$ is embedded to an appropriate Lebesgue space.
      \item  Under suitable assumptions, one can extend the theory on non--homogeneous fluxes $G=G(X,U)$ and equation \eqref{eq:balance} with non--zero right--hand side $h=h(X,U)$.
      
        \item Due to the definition of weak solutions, it is enough to consider the integrability and regularity of $U$ only locally in $\CX$.
    \end{itemize}
\end{rem}

 Due to the assumption on the convexity of $\CO$, Theorem \ref{thm:bes} could be straightforwardly deduced from \cite{Feireisletal};
       however, for the reader's convenience, we present the proof in Section \ref{sec:main}.        
       It is worth noting that the convexity of $\CO$ might not be natural for all applications (this is e.g.\,the case of the polyconvex elasticity, see Section \ref{sec:appl}.
      To this purpose, we present a theorem dealing with the case of non--convex $\CO$.

\begin{thm}
    \label{thm:non_convex}
       Let the assumptions of  Theorem \ref{thm:bes} be satisfied, but instead of \eqref{eq:assumpt_convex}
       we assume that
       \begin{equation}
        \label{eq:assumpt_compact}
           \mbox{the essential range of $U$ is compact in $\CO$.}
       \end{equation}
       Then $U$ is a weak solution of the companion law \eqref{eq:companion} with the flux $Q$.
\end{thm}

Apparently, the conclusions of the previous theorems are reasonably weaker in comparison with some known results for particular conservation laws. As an example, the result of Constantin et al. in \cite{ConstETiti} does not need the Besov--type regularity with respect to time. Having more knowledge about the nonlinear part of $G$, we may be able to  relax the class of solutions in Theorem \ref{thm:bes}, what is discussed in Section \ref{sec:appl}.

Finally, we observe that in case we consider hyperbolic systems, the opposite direction of the Onsager's hypothesis  is almost trivial. This is of course completely different situation than the case of incompressible Euler system, which is not a hyperbolic conservation law and the construction of solutions dissipating the energy was a challenge. 
It is well known, cf.~\cite[Chapter 1]{dafermos} among others, that 
 shock solutions dissipate energy.  Following Dafermos again, we note that 
 crucial  properties of  local behavior of shocks
may be investigated, without loss of generality, within the framework of systems
in one-space dimension. Thus the essence can be already seen even on a simple example of the Burger's equation $u_t+(u^2/2)_x=0$. Classical solutions also satisfy $(u^2/2)_t+(u^3/3)_x=0$, which can be considered as a companion law.  The shock solutions to the equation in the first form satsify Rankine-Hugoniot condition $s(u_{l}-u{r})=(u^2_l-u^2_r)/2$, thus the speed of the shock is $
 s=(u_l+u_r)/2$, where $u_l=\lim_{y\to x(t)^-}u(y,t)$ and $u_r$ is defined correspondingly. 
 Considering the second equation one gets $s=2(u_l^2+u_lu_r+u_r^2)/3(u_l+u_r)$, which in an obvious way is different. More generally if we multiply~\eqref{eq:balance} with the function ${\mathcal B}$ one easily concludes that to provide that Rankine-Hugoniot conditions to be satisfied for the companion law,  we end up with a trivial companion law, namely ${\mathcal B}\equiv const$.
 
  Thus, knowing the regularity of shock solutions,  as was shown in \cite[Prop. 2.1]{Feireisletal}
 $$(BV\cap L^\infty)(\Omega)\subset B_{\infty}^{\frac{1}{q}}(\Omega)$$ for every $q\in[1,+\infty]$,
 we observe that  our assumptions are sharp. 
 

Let us briefly mention the outline of the rest of the paper. In Section \ref{sec:besov}, we introduce the notation. Section \ref{sec:main} contains proofs of the main propositions. 
Section \ref{sec:appl} is devoted to some relaxation of the conditions in Theorem \ref{thm:bes} and applications of the main theorems are also presented. 

\section{Notation and auxiliary estimates}
 \label{sec:besov}
 
We will briefly present some properties of the Besov spaces $B^{\alpha}_{q,\infty}$.
Let $\CX$ be as above, $\alpha\in(0,1)$ and $q\in[1,\infty)$. 
We denote by $B^{\alpha}_{q,\infty}(\CX)$ the Besov space which is defined as follows
\begin{equation*}
    B^{\alpha}_{q,\infty}(\CX)
    =   
    \left\{
            U\in L^q(\CX) \colon \quad
            |U|_{B^{\alpha}_{q, \infty}(\CX)} <\infty
    \right\}
\end{equation*}
with
\begin{align*}
    |U|_{B^{\alpha}_{q,\infty}(\CX)}
    =
    \sup_{\xi\in \BR^k}\frac{\|U(\cdot)-U(\cdot-\xi)\|_{L^q(\CX\cap (\CX+\xi))}}{|\xi|^{\alpha}}.
\end{align*}
On $B^{\alpha}_{q,\infty}(\CX)$ we consider the standard norm
\begin{equation*}
    \|U\|^q_{B^{\alpha}_{p,\infty}(\CX)}
    =
    \|U\|^q_{L^q(\CX)}+|U|^q_{B^{\alpha}_{q,\infty}(\CX)}.
\end{equation*}
Assume that a non--negative function $\eta_1 \in C^{\infty}(\BR^k)$ has a compact support in $B(0,1)$ and $\int_{\BR^k}\eta_1(X)\de{X}=1$. For $\eps>0$ we denote $\eta_{\eps}(X) = \frac{1}{\eps^k}\eta_1(\frac{X}{\eps})$ and
\begin{equation*}
    [f]_{\eps}(X)=f\ast \eta_{\eps}(X)
\end{equation*}
which is defined at least in $\CX_{\eps}= \{X\in \CX \colon \dist(X,\partial \CX)>\eps\}$. For vector or matrix--valued functions  the convolution is defined component--wise.
For $\CK \subseteq \BR^k$ and $\delta>0$ we also use the notation 
\begin{equation*}
    \CK^{\delta}=\{X\in \BR^k \colon \dist(X,\CK)<\delta\}=\cup_{X\in \CK}B(X,\delta).
\end{equation*}
One easily shows that  for $f\in B^{\alpha}_{q,\infty}(\CX)$ the following estimates hold
    \begin{align}
     \label{eq:mollif_nabla}
        \|D_X [f]_{\eps}\|_{L^q(\CX_{\eps})}&\leq C \|f\|_{B^{\alpha}_{q,\infty}(\CX)}\eps^{\alpha-1},
    \\
     \label{eq:mollif_diff}
        \|[f]_{\eps}-f\|_{L^q(\CX_{\eps})}&\leq C\|f\|_{B^{\alpha}_{q,\infty}(\CX) }\eps^{\alpha},
    \\
     \label{eq:transport}
        \|f(\cdot-y)-f(\cdot)\|_{L^q(\CX\cap(\CX+y))}&\leq C \|f\|_{B^{\alpha}_{q,\infty}(\CX)}|y|^{\alpha}
    \end{align}
where $C$ depends only on $\CX$.

\section{The proof of the main results}
\label{sec:main}
In what follows, we will denote by $C$ a constant 
independent of $\eps$.

\subsection{Commutator estimates}
\label{subsec:convex}
The essential part of  the proof of  Theorem \ref{thm:bes} pertains  the estimation of the  nonlinear commutator
\begin{equation*}
    [G(U)]_{\eps}-G([U]_{\eps}).
\end{equation*}
It is based on the following observation, which appears in a special form in \cite{Feireisletal}. The rest of the proof of Theorem \ref{thm:bes} is a reminescence of the paper of \cite{ConstETiti}.
\begin{lem}
    \label{lem:nonlin_commut}
     Let $\CO$ be a convex set, $U\in L^2_{loc}(\CX,\CO)$, $G\in C^2(\CO;\BR^n)$ and let
     \begin{equation}
      \label{eq:sec_der_bdd}
        \sup_{i,j \in{1,\dots,d}}\|\partial_{U_i}\partial_{U_j} G(U)\|_{L^{\infty}(\CO)}<+\infty.
    \end{equation}
    Then there exists $C>0$ depending only on $\eta_1$, second derivatives of $G$ and $k$ (dimension of $\CO$) such that 
    \begin{equation*}
        \left\|
            [G(U)]_{\eps}-G([U]_{\eps})
        \right\|_{L^{q}(K)}
        \leq C\Big(\|[U]_{\eps}-U\|^2_{L^{2q}(K)}
        +   \sup_{Y\in \supp \eta_{\eps}}\|U(\cdot)-U(\cdot-Y)\|^2_{L^{2q}(K)}
        \Big)
    \end{equation*}
    for $q\in [1,\infty)$, where $K\subseteq \CX$ satisfies $K^{\eps} \subseteq \CX$.
\end{lem}
\begin{proof}
    Without loss of generality, we assume that  $G$ is a scalar function and $U$ is  finite everywhere on $\CX$.
    Then, because of \eqref{eq:sec_der_bdd} 
    we get for $X$, $Y\in K$
    \begin{align}
        \label{eq:first_tay}
        \left|G(U(X))-G([U]_{\eps}(X))-D_U G\circ U(X)(U(X)-[U]_{\eps}(X))\right|
        &\leq 
        C |U(X)-[U]_{\eps}(X)|^2,
        \\
        \label{eq:sec_tay}
        \left|G(U(X))-G(U(Y))-D_U G\circ U(X)(U(X)-U(Y))\right|
        &\leq 
        C |U(X)-U(Y)|^2.      
    \end{align}
    We convolve  \eqref{eq:sec_tay} with $\eta_{\eps}$ in variable $Y$ and apply Jensen's inequality on the left--hand side
    \begin{equation}
        \label{eq:convol_tayl}
         \left|G(U(X))-[G(U)]_{\eps}(X)-D_U G\circ U(X)(U(X)-[U]_{\eps}(X))\right|
        \leq 
        C |U(X)-U(\cdot)|^2\ast_Y \eta_{\eps}.
    \end{equation}
    Finally, coupling \eqref{eq:first_tay} and \eqref{eq:convol_tayl} implies to
    \begin{equation}
        \label{eq:eq3}
        \left| 
            G([U]_{\eps}(X))-[G(U)]_{\eps}(X)
        \right|
        \leq
        C\left(|U(X)-[U]_{\eps}(X)|^2+|U(X)-U(\cdot)|^2\ast_Y \eta_{\eps}(X)\right).
    \end{equation}
    In order to complete the proof, we 
    use  Jensen's inequality to estimate the $L^q$ norm of the second 
    term on the right--hand side of \eqref{eq:eq3}
        \begin{align*}
        &\int_{K} \left| \int_{\supp \eta_{\eps}}|U(X)-U(X-Y)|^2\eta_{\eps}(Y)\de{Y}\right|^q\de{X}
        \\
        & \ \leq
        \int_{\supp \eta_{\eps}} \int_{K}|U(X)-U(X-Y)|^{2q}\eta_{\eps}(Y)\de{X}\de{Y}
        \leq
         \sup_{Y\in \supp \eta_{\eps}}\|U(\cdot)-U(\cdot-Y)\|^{2q}_{L^{2q}(K)}.
    \end{align*}   
\end{proof}

\subsection{Proof of Theorem \ref{thm:bes}}
    \label{subsec:proof1}
    Let $\eps_0>0$ and consider a test function
    $\psi\in C^{\infty}(\CX)$ such
    that $\supp \psi\subseteq \CX_{\eps_0}$. 
    Mollifying \eqref{eq:balance} 
    by $\eta_{\eps}$, we obtain 
    \begin{equation}
        \label{eq:moll_bal}
         \diver_{X} [G(U)]_{\eps}=0 \quad \mbox{in } \CX_{\eps_0}
    \end{equation}
    whenever $\eps<\eps_0$.
    We multiply  both sides of \eqref{eq:moll_bal} by $\psi {\mathcal B}([U]_{\eps})$ (where ${\mathcal B}$ comes from \eqref{eq:algebraic_cons}) from the left and get
    \begin{equation}
        \label{eq:prep}
         \int_{\CX}\psi(X){\mathcal B}([U]_{\eps}(X)) \diver_{X} ([G(U)]_{\eps}(X)) \de{X}=0.
    \end{equation}
    We can recast the previous equality as follows
    \begin{equation*}
        \int_{\CX}\psi(X){\mathcal B}([U]_{\eps}(X)) \diver_{X} G([U]_{\eps}(X)) \de{X}= \int_{\CX}R_{\eps}\de{X}
    \end{equation*}
    with the commutator
    \begin{equation}\label{error}
        R_{\eps}=\psi(X){\mathcal B}([U]_{\eps}(X))
        \diver_{X}
        \Big(
            G([U]_{\eps}(X))-[G(U)]_{\eps}(X)
        \Big).
    \end{equation} 
    Due to \eqref{eq:algebraic_cons}, equality \eqref{eq:prep} might be adjusted to the form
    \begin{equation}
        \label{eq:mollifRendef}
        -\int_{\CX} Q([U]_{\eps}(X)) (D_X\psi(X))^T \de{X}
        = 
        \int_{\CX}R_{\eps}\de{X}.
    \end{equation}
 In order to show that the right--hand side of \eqref{eq:mollifRendef} converges to zero as $\varepsilon\to0$, we write
    \begin{align}\nonumber
        \label{eq:commut_full}
        \int_{\CX}R_{\eps}(X)\de{X}&=\int_{\CX} 
       \Big(
            G([U]_{\eps})-[G(U)]_{\eps}
        \Big):\left((D_U {\mathcal B}^T)([U]_{\eps})D_X [U]_{\eps} \psi\right)\de{X}
        \\ 
        &+ 
        \int_{\CX}
        \Big(
            G([U]_{\eps})-[G(U)]_{\eps}
        \Big)\colon\left({\mathcal B}^T([U]_{\eps})D_{X}\psi\right)\de{X}
        \\\nonumber
        &=I^{1}_{\eps}+I^{2}_{\eps}.
    \end{align}
    The first integral is estimated using Lemma \ref{lem:nonlin_commut} and \eqref{eq:mollif_nabla} as follows
    \begin{align*}
        |I^2_{\eps}|&\leq
        C\|{\mathcal B}\|_{W^{1,\infty}(\CO)}\|D_X [U]_{\eps}\|_{L^3(\CX_{\eps_0})}\|[U]_{\eps}-U\|^2_{L^3(\CX_{\eps_0})}\|\psi\|_{W^{1,\infty}(\CX_{\eps_0})}
        \\
        &\leq C \eps^{\alpha-1}\eps^{2\alpha}
    \end{align*}
    Similarly, we have
    \begin{equation*}
        |I^1_{\eps}|\leq C \eps^{\alpha},
    \end{equation*}
    hence, 
    \begin{equation*}
        \int_{\CX}R_{\eps}\de{X}\to 0 \quad \mbox{as $\eps\to 0$}
    \end{equation*}
    as long as $\alpha>\frac{1}{3}$.
    
   The convergence of the left--hand side of \eqref{eq:mollifRendef} follows from  the Vitali theorem.  Indeed, 
   the equi-integrability of $Q([U]_{\eps})$ in $\CX_{\eps_0}$ is a consequence of that of $|[U]_{\eps}|^3$ and the growth conditions on $Q$.

\begin{rem}
    Having $\CO$ non--convex, we face the problem that $[U]_{\eps}$ does not have to belong to $\CO$. 
    The convexity was crucial to conduct the Taylor expansion argument in Lemma \ref{lem:nonlin_commut}. 
    However, we will see that a suitable extension of functions $G$, ${\mathcal B}$ and $Q$ does not alter the previous proof significantly. 
\end{rem}

\subsection{Proof of Theorem \ref{thm:non_convex}}
\label{subsec:proof2}
There exists $\delta>0$ depending only on $\CK$ and $\CO$ such that $\CK^{2\delta}\subseteq \CO$. Let $\tilde{G}\in C^2(\BR^n;\BM^{(k+1)\times n})$, $\tilde{{\mathcal B}}\in C^1(\BR^n;\BM^{1\times n})$ and $\tilde{Q}\in C(\BR^n,\BM^{1\times (k+1)})$ be compactly supported functions satisfying
$\tilde{G}=G$, $\tilde{{\mathcal B}}={\mathcal B}$ and $\tilde{Q}=Q$  on $\CK^{\delta}$. Such functions exist as there is a set $\CR$ with a smooth boundary satisfying $\CK^{\delta}\subseteq \CR \subseteq \CO$. Thus, relation~\eqref{eq:algebraic_cons} holds also for $G$, ${\mathcal B}$ and $Q$ on $\CK^{\delta}$.
    
Similarly to the proof of Theorem \ref{thm:bes}, for a function $\psi \in C^{\infty}(\CX)$ compactly supported in $\CX_{\eps_0}$, we obtain for $\varepsilon<\varepsilon_0$
\begin{equation*}
    \int_{\CX}\psi\tilde{{\mathcal B}}([U]_{\eps}) \diver_{X} [\tilde{G}(U)]_{\eps} \de{X}=0.
\end{equation*}
We can write the previous equality as follows
\begin{equation}
    \label{eq:molif}
     \int_{\CX}\psi\tilde{{\mathcal B}}([U]_{\eps}) \diver_{X} \tilde{G}([U]_{\eps}) \de{X}= \int_{\CX}\tilde{R}_{\eps}\de{X}
\end{equation}
with the commutator
\begin{equation}
    \label{eq:comm_tilda}
     \tilde{R}_{\eps}=\psi\tilde{{\mathcal B}}([U]_{\eps})
     \diver_{X}\Big(
        \tilde{G}([U]_{\eps})-[\tilde{G}(U)]_{\eps}
     \Big).
\end{equation} 
Analogously to Subsection \ref{subsec:proof1}, $\int_{\CX}\tilde{R}_{\eps}\de{X}$ vanishes as $\eps \to 0$ due to Lemma \ref{lem:nonlin_commut}; hence, we may turn our attention to the left--hand side of  \eqref{eq:molif}. We show that it converges to 
\begin{equation}
    -\int_{\CX}Q(U) (D_X\psi)^T \de{X}.
\end{equation}
To this end, we put 
\begin{equation*}
    \CG_{\eps}^\delta = \{X\in \CX \colon |U(X)-[U]_{\eps}(X)|<\delta\}
\end{equation*}
and since
$D_U \tilde{Q}_{j}([U]_{\eps})=
    \tilde{{\mathcal B}}([U]_{\eps})D_U \tilde{G}_{j}([U]_{\eps})$ on  $\CG_{\eps}^\delta$ we obtain 
\begin{align*}
    &\left|
        \int_{\CX}\psi\tilde{{\mathcal B}}([U]_{\eps}) \diver_{X} \tilde{G}([U]_{\eps}) \de{X}+\int_{\CX} Q(U)(D_X\psi)^T \de{X}
    \right|
    \\ & \quad
    \leq 
    \left|
        \int_{\CX \backslash\CG_{\eps}^\delta}
        \psi\tilde{{\mathcal B}}([U]_{\eps}) \diver_{X} \tilde{G}([U]_{\eps}) \de{X}
    \right|
    +
    \left|
        \int_{\CX \backslash\CG_{\eps}^\delta}
         Q(U)(D_X \psi)^T\de{X}
        \right|
        \\ & \quad
        +\left|\int_{\CG_{\eps}^\delta}(\tilde{Q}(U)-\tilde{Q}([U]_{\eps}))(D_X \psi)^T\de{X}
    \right|
    =
    I^1_{\eps}+I^2_{\eps}+I^3_{\eps}.
\end{align*}
   To estimate $I^1_{\eps}$, recall that $\tilde{G}$ and $\tilde{{\mathcal B}}$ are compactly supported, therefore
\begin{align*}
    I^1_{\eps}\leq     
    \int_{\CX \backslash\CG_{\eps}^\delta}
    \left|
    \psi\tilde{{\mathcal B}}([U]_{\eps}) D_U\tilde{G}([U]_{\eps}) D_{X} [U]_{\eps} \right|\de{X}
    \leq
    C\|\psi\|_{C^1} \int_{\CX \backslash\CG_{\eps}^\delta}|D_X [U]_{\eps}|\de{X}.
\end{align*}
By the means of H\"older's and Chebyshev's inequality, \eqref{eq:mollif_nabla} and \eqref{eq:mollif_diff} we observe that
\begin{align*}
    I^1_{\eps}&\leq 
    C\|\psi\|_{C^1} \|D_X [U]_{\eps}\|_{L^3(\CX_{\eps_0})} \left|
    \CX\backslash \CG_{\eps}^\delta 
    \right|^{\frac{2}{3}}
    =
    C\|\psi\|_{C^1} \|D_X [U]_{\eps}\|_{L^3(\CX_{\eps_0})}
    \left|\{
        |U-[U]_{\eps}|>\delta
        \}
    \right|^{\frac{2}{3}}
    \\&
    \leq \frac{ C\|\psi\|_{C^1}}{\delta^{2}}
    \|D_X [U]_{\eps}\|_{L^3(\CX_{\eps_0})}
    \|U-[U]_{\eps}\|^2_{L^3(\CX_{\eps_0})}\leq \frac{ C\|\psi\|_{C^1}}{\delta^{2}} \eps^{3\alpha -1}.
\end{align*}
The integral $I^2_{\eps}$ vanishes, as $\|Q(U)\|_{L^{\infty}(\CX)}<\infty$.
Finally, we observe that
\begin{align*}
    I^3_{\eps}\leq \|\psi\|_{C^1}
    \int_{\CX_{\eps_0}}|\tilde{Q}(U)-\tilde{Q}([U]_{\eps})|\de{X}.
\end{align*}
Therefore, $I^3_{\eps}\to 0$ due to the almost everywhere convergence of $\tilde{Q}(U)-\tilde{Q}([U]_{\eps})$ to zero and boundedness of $\tilde{Q}$. 

\section{Applications}
\label{sec:appl}
Observe that we have considered so far genuinely nonlinear fluxes $G$. 
The key part of the proof was to estimate
    \begin{equation}
       \int_{\CX} 
       \Big(
            G([U]_{\eps})-[G(U)]_{\eps}
        \Big):\left((D_U {\mathcal B}^T)([U]_{\eps})D_X [U]_{\eps} \psi\right)\de{X},
            \end{equation}
where the integral vanishes whenever $G$ is an affine. Using this observation we might expect to drop some conditions on $U$ in the main theorems if some components of $G$ are affine functions.

We present three extensions of Theorem \ref{thm:bes}, which follow directly from the previous observation. The first gives a sufficient condition to drop the Besov regularity with respect to some variables. It is connected with the columns of $G$.
\begin{cor}
    \label{cor1}
    Let $G=(G_{1},\dots,G_{s},G_{ s+1},\dots G_{ k})$ where $G_{ 1},\dots,G_{ s}$ are affine vector--valued functions and $\CX = \CY\times \CZ$ where $\CY\subseteq \BR^s$ and $\CZ\subseteq \BR^{k+1-s}$. Then it is enough  to assume that $U\in L^3(\CY;B^{\alpha}_{3,\infty}(\CZ))$ in Theorem \ref{thm:bes}.
\end{cor}

Next, we specify when we can omit  the Besov regularity with respect to some components of $U$.
\begin{cor}
    \label{cor2}
    Assume that $U=(V_1,V_2)$ where $V_1=(U_1,...,U_s)$ and $V_2=(U_{s+1},\dots,U_n)$. If ${\mathcal B}$ does not depend on $V_1$ and $G=G(V_1,V_2)=G_1(V_1)+G_2(V_2)$ and $G_1$ is linear then it is enough to assume $U_1,\dots,U_s \in L^3(\CX)$ in Theorem \ref{thm:bes}. 
\end{cor}

Finally, we deal with the case when some components of ${\mathcal B}$ are not Lipschitz on $\CO$, but appropriate rows of $G$ are affine functions.

\begin{cor}
    \label{cor3}
    Assume that a $j$--th row of $G$ is an affine function. Then the statement of Theorem \ref{thm:bes} holds even if we assume that ${\mathcal B}_{j}$ is only locally Lipschitz in $\CO$.
\end{cor}

In the rest of this paper, we present a few examples on which the general theory applies. Some of them show how the general framework allows to recover some known results.  In what follows, we consider $\CX = (0,T)\times \BT^3$, $X=(t,x)$ and $\alpha>\frac{1}{3}$. We also present the systems in their standard form denoting $\nabla_x$ and $\diver_x$ the correspondent operators with respect to the spatial coordinate $x$.

\subsection{Incompressible Euler system}
\label{subsec:incom_eul}
Let us consider the system of equations
\begin{align*}
    \left.\begin{aligned}\diver_{x}\bu^T &= 0\\ 
    \partial_t \bu + (\bu\cdot \nabla_x) \bu + \nabla_x p &= 0
    \end{aligned} 
    \right\}
    \quad \mbox{in } \CX
\end{align*}
for an unknown vector field $\bu\colon (0,T)\times \BT^3 \to \BR^3$
and scalar  $p\colon (0,T)\times \BT^3 \to \BR$.
The system can be rewritten into the divergence form  with respect to $X=(t,x)$
\begin{align}
 \label{eq:eul_div_form}
    \left.
    \begin{aligned}
        \diver_{x}\bu^T &= 0, \\
            \partial_t \bu +\diver_x(\bu\otimes\bu +  p\BI) &= 0.
    \end{aligned}
    \right\}
\end{align}
By multiplying \eqref{eq:eul_div_form} with ${\mathcal B}(p,\bu)= (p-1\slash 2 |\bu|^2,\bu^T)$ we obtain the conservation law for the  energy
\begin{equation}
 \label{eq:inc_eul}
    \partial_t \left(\frac{1}{2}|\bu|^2\right)
    + \diver_x \left(\frac{1}{2}|\bu|^2 + p\bu^T \right)=0.
\end{equation}
    
Corollaries \ref{cor1}, \ref{cor2} and \ref{cor3} imply that any weak solution $(p,\bu)\in  L^3(\CX)\times L^3(0,T;B^{\alpha}_{3,\infty}(\BT^3))$ is a weak solution to \eqref{eq:inc_eul}.
    
\begin{rem}
    This result is comparable to \cite{ConstETiti}.
\end{rem}
    
\subsection{Compressible Euler system}
    \label{subsec:comp_eul}
We consider the compressible Euler equations in the following form
\begin{align}
 \label{eq:comp_eul}
    \left.
    \begin{aligned}
        \partial_t \rho+ \diver_{x}(\rho\bu^T) &= 0 \\ 
        \partial_t \bu + \diver_x(\bu\otimes \bu) + \frac{\nabla_x p(\rho)}{\rho} &= 0
    \end{aligned} \right\}
    \quad \mbox{in } \CX
\end{align}
for an unknown vector field $\bu\colon \CX \to \BR^3$
and scalar  $\rho\colon \CX \to \BR$. The function $p\colon [0,\infty)\to\BR$ is given.  Let $P$ be a primitive function to $\frac{p(\rho)}{\rho}$ such that $P(1)=0$. Then the system can be rewritten into the divergence form
\begin{align}
 \label{eq:comp_eul_div}
    \left.
    \begin{aligned}
        \partial_t \rho+ \diver_{x}(\rho\bu^T) &= 0,\\ 
        \partial_t \bu + \diver_x\left(\bu\otimes \bu+P(\rho)\BI\right)  &= 0.         
    \end{aligned}
    \right.
\end{align}
To get the conservation of the energy, we multiply \eqref{eq:comp_eul_div} with
\begin{equation*}
    {\mathcal B}(\rho,\bu)=\left( P(\rho)+ \rho P'(\rho)-\frac{1}{2}|\bu|^2, \rho \bu^T\right)
\end{equation*}
and obtain
\begin{equation}
 \label{eq:comp_cons}
    \partial_t\left(
        \frac{1}{2}\rho |\bu|^2 + \rho P(\rho)
              \right)
    +\diver_x\left[
        \left(
            \frac{1}{2}\rho |\bu|^2
            +\rho P(\rho)+p(\rho)
        \right)
    \bu^T\right]
    =0
\end{equation}

Let $(\rho,\bu)\in L^3(0,T;B^{\alpha}_{3, \infty}(\BT^3))\times L^3(0,T;B^{\alpha}_{3, \infty}(\BT^3;\BR^3))$ be a weak solution to \eqref{eq:comp_eul_div} such that $\rho\in [\ubar{\rho},\bar{\rho}]$ for some $0<\ubar{\rho}<\bar{\rho}<\infty$ and $\bu\in B(0,R)$ for some $R>0$. Moreover, if $p\in C^2([\ubar{\rho},\bar{\rho}])$, we use\footnote{We can extend $p$ from $[\ubar{\rho},\bar{\rho}]$ on $\BR$ such that the extended function will be of class $C^2$ and compactly supported in $\BR$. Moreover, due to the boundedness of $|\bu|$ we can write $|\bu|^2 = \bu\cdot T(\bu)$ in $\CX$ where $T$ is a bounded Lipschitz function on $\BR^3$.} Corollary \ref{cor1} to show that $(\rho,\bu)$ is a weak solution to \eqref{eq:comp_cons}. In the contrast with the incompressible case, the continuity equation (the first equation of \eqref{eq:comp_eul}) is not linear with respect to $\rho$ and $\bu$. Therefore, we have to assume that $\bu$ is bounded to provide ${\mathcal B}(\rho,\bu)$ is Lip
 schitz o
 n the range of $(
 \rho,\bu)$.
    
\begin{rem}
   We have considered the formulation of the compressible Euler system with the time derivative over a linear function of $(\rho, \bu)$. This has lead to a slightly different sufficient condition in comparison to \cite{Feireisletal}.
\end{rem}
    
\begin{rem}
    If $\rho>0$, system \eqref{eq:comp_eul} can be rewritten with respect to the quantities $\rho$ and $\bm = \rho \bu$ as follows
    \begin{align}
     \label{eq:comp_eul_div_2}
    \left.
    \begin{aligned}
        \partial_t \rho+ \diver_{x}(\bm) &= 0\\ 
        \partial_t \bm + \diver_x\left(\frac{\bm\otimes \bm}{\rho}+p(\rho)\BI\right)  &= 0           
    \end{aligned}
    \right\}\quad \mbox{in } \CX.
\end{align}
A suitable choice of  ${\mathcal B}$ is then \begin{equation}
    {\mathcal B}(\rho,\bm) = \left(P(\rho)+ \rho P'(\rho)-\frac{|\bm|^2}{2\rho^2},\frac{\bm^T}{\rho}\right),
\end{equation}
which leads to the companion law
\begin{equation}
    \label{eq:comp_cons_2}
    \partial_t\left(
        \frac{|\bm|^2}{2\rho}  + \rho P(\rho)
              \right)
    +\diver_x\left[
        \left(
            \frac{|\bm|^2}{2\rho}
            +\rho P(\rho)+p(\rho)
        \right)
    \bu\right]
    =0.
\end{equation}
As the continuity equation is now linear with respect to $(\rho,\bm)$, we can apply Corollaries \ref{cor1} and \ref{cor3}. As their consequence, 
a weak solution 
\[(\rho,\bm)\in L^3(0,T;B^{\alpha}_{3, \infty}(\BT^3))\times L^3(0,T;B^{\alpha}_{3, \infty}(\BT^3;\BR^3))
\]
such that $\rho\in [\ubar{\rho},\bar{\rho}]$ for some $0<\ubar{\rho}<\bar{\rho}<\infty$
is also a weak solution to \eqref{eq:comp_cons_2}.

\end{rem}
\subsection{Polyconvex elasticity}
 \label{subsec:polyconv}
Let us consider the evolution equations of nonlinear elasticity, see e.g.\,\cite{Dafermos1985} or \cite{Demoulini2001},
\begin{align}
    \label{eq:non_elast}
    \left.
    \begin{aligned}
        \partial_t F &= \nabla_x \bv \\
        \partial_t \bv &= \diver_x \left(D_{F}W(F)\right) 
    \end{aligned}
    \right\}
    \quad \mbox{in } \CX,
\end{align}
for an unknown matrix field $F\colon \CX \to \BM^{k\times k}$, 
and an unknown vector field $\bv\colon \CX \to\BR^k$. Function 
$W\colon \CU \to \BR$ is given.
For many applications, $\CU= \BM_+^{k\times k}$ 
where $\BM_+^{k\times k}$ denotes the subset of $\BM^{k\times k}$ containing only matrices having positive determinant, see e.g.\,\cite{ball_open_prob} for the discussion on the form of $W$ and $\CU$. 
Let us point out that $\BM_+^{k\times k}$ is a non--convex connected set. 
    
System \eqref{eq:non_elast} can be rewritten into the divergence form in $(t,x)$ as follows
\begin{align}
 \label{eq:polyconv_div}
    \left.
    \begin{aligned}
        \partial_t F_{i,j} &= \partial_{x_i}u_j      =\diver_{x}\left(\left(\mathbf{e}^i\right)^T u_j\right),
        \quad \mathbf{e}^i_j=\delta_{i,j},
        \\
        \partial_t \bv &= \diver_x \left(D_{F}W(F)\right)^T.
    \end{aligned}
  \right.
\end{align}
By considering $F$ to have values in $\BR^{k^2}$ and taking ${\mathcal B}(F,\bv)=(\{D_F W(F)\}^T,\bv^T)$, we obtain the companion law
\begin{equation}
 \label{eq:energy_polyc}
    \partial_t \left(\frac{1}{2}|\mathbf{v}|^2+W(F)\right)
    -\diver\left( D_{F}W(F)\mathbf{v} \right)=0.
\end{equation}
    
Let $(F,\bv)\in B^{\alpha}_{3,\infty}(\CX;\BM^{k\times k})\times B^{\alpha}_{3,\infty}(\CX;\BR^3)$ be a weak solution to \eqref{eq:polyconv_div} such that $F$ has a compact range in $\CU$ and $\bv$ in $\BR^k$. Directly from Theorem \ref{thm:non_convex}, $(F,\bv)$ is a weak solution to \eqref{eq:energy_polyc} whenever $W\in C^3(\CU)$.

Note that this observation for polyconvex elasticity is up to our best knowledge an original contribution. 

\subsection{Magnetohydrodynamics}
\label{subsec:MHD}
Let us consider the system
\begin{align}
    \label{eq:MHD}
    \left.
    \begin{aligned}
    \diver_{x}\bu^T &= 0
            \\
            \diver_{x}\bh^T &= 0 
            \\
            \partial_t \bu + (\bu\cdot \nabla_x) \bu + \nabla_x p &= (\curl_x \bh)\times\bh  
            \\
            \partial_t \bh + \curl_x(\bh \times \bu)  &= 0 
    \end{aligned}
   \right\}     
   \quad &\mbox{in } \CX
\end{align} 
for unknown vector functions $\bu\colon \CX \to \BR^3$ and $\bh\colon \CX \to \BR^3$ and an unknown scalar function $p\colon \CX\to \BR$. It describes the motion of an ideal electrically conducting fluid, see e.g. \cite[Chapter VIII]{landau}.
Using standard vector calculus identities, \eqref{eq:MHD} can be written in the divergence form as follows:
\begin{align*}
    \diver_{x}\bu^T &= 0,\\
    \diver_{x}\bh^T &= 0,\\
    \partial_t \bu + \diver_x\left(\bu\otimes \bu + p\BI +\frac{1}{2}|\bh|^2\BI-\bh\otimes \bh\right)&= 0,\\
    \partial_t \bh + \diver_x(\bh \otimes \bu-\bu\otimes \bh)  &= 0.
\end{align*}
        
With ${\mathcal B}(p,\bu,\bh)=(p-1\slash 2 |\bu|^2,-\bh\cdot\bu,\bu^T,\bh^T)$, the conservation of the total energy reads:    
\begin{align}
    \label{eq:MHD_en}
    &\partial_t\left(
            \frac{1}{2}|\mathbf{u}|^2 + \frac{1}{2}|\bh|^2
    \right)+\diver_x\left[
    \left(\frac{1}{2} |\mathbf{u}|^2+p+|\bh|^2\right) \mathbf{u}^T -(\bu\cdot\bh)\bh^T 
           \right]=0.
\end{align}
    A combination of Corollaries \ref{cor1}, \ref{cor2} and \ref{cor3} implies that any weak solution 
\begin{equation}
    (p,\bu,\bh)\in L^3(\CX)\times\Big(L^3(0,T;B^{\alpha}_{3,\infty}(\BT^3))\Big)^2
\end{equation} 
is a weak solution to \eqref{eq:inc_eul}. A similar result was obtained e.g. in \cite{Cafetal}.
    
 \subsection{Further examples}  
   The list of examples is still far from being complete, however it is not our goal, and surely not an expectation of a reader, to provide an extended list. Among numerous further examples we will only mention inviscid compressible magneto--hydrodynamics. A direct combination of Subsection \ref{subsec:comp_eul} and \ref{subsec:MHD} gives  a sufficient condition to satisfy the relevant energy equality.
Another worth of mentioning example is heat conducting gas, see also \cite{Drivas:2017aa}.

\medskip

\noindent
{\bf Acknowledgements}
This work was partially supported by the Simons - Foundation grant 346300 and the Polish Government MNiSW 2015-2019 matching fund.  P.G. and A. \'S.-G. received support from the National Science Centre (Poland), 
2015/18/MST1/00075.
The research of M. M. leading to these results has received funding from the European Research Council under the European Union's Seventh Framework Programme (FP7/2007-2013)/ ERC Grant Agreement 320078. The Institute of Mathematics of the Academy of Sciences of the Czech Republic is supported by RVO:67985840.

\bibliographystyle{abbrv}    
\bibliography{reference}

\begin{thebibliography}{10}

\bibitem{ball_open_prob}
J.~M. Ball.
\newblock Some open problems in elasticity.
\newblock In {\em Geometry, mechanics, and dynamics}, pages 3--59. Springer,
  New York, 2002.

\bibitem{Buc15}
T.~Buckmaster.
\newblock Onsager's conjecture almost everywhere in time.
\newblock {\em Comm. Math. Phys.}, 333(3):1175--1198, 2015.

\bibitem{BSLISJ15}
T.~Buckmaster, C.~De~Lellis, P.~Isett, and L.~Sz{\'e}kelyhidi, Jr.
\newblock Anomalous dissipation for {$1/5$}-{H}{\"o}lder {E}uler flows.
\newblock {\em Ann. of Math. (2)}, 182(1):127--172, 2015.

\bibitem{BDKS16}
T.~Buckmaster, C.~De~Lellis, and L.~Sz{\'e}kelyhidi, Jr.
\newblock Dissipative {E}uler flows with {O}nsager-critical spatial regularity.
\newblock {\em Comm. Pure Appl. Math.}, 69(9):1613--1670, 2016.

\bibitem{delell_rec}
T.~Buckmaster, C.~De~Lellis, L.~Sz\'ekelyhidi, Jr., and V.~Vicol.
\newblock Onsager's conjecture for admissible weak solutions.
\newblock {\em arxiv}, (1701.08678), 2017.

\bibitem{Cafetal}
R.~E. Caflisch, I.~Klapper, and G.~Steele.
\newblock Remarks on singularities, dimension and energy dissipation for ideal
  hydrodynamics and {MHD}.
\newblock {\em Comm. Math. Phys.}, 184(2):443--455, 1997.

\bibitem{cheskidov}
A.~Cheskidov, P.~Constantin, S.~Friedlander, and R.~Shvydkoy.
\newblock Energy conservation and {O}nsager's conjecture for the {E}uler
  equations.
\newblock {\em Nonlinearity}, 21(6):1233--1252, 2008.

\bibitem{ConstETiti}
P.~Constantin, W.~E, and E.~S. Titi.
\newblock Onsager's conjecture on the energy conservation for solutions of
  {E}uler's equation.
\newblock {\em Comm. Math. Phys.}, 165(1):207--209, 1994.

\bibitem{dafermos}
C.~M. Dafermos.
\newblock {\em Hyperbolic conservation laws in continuum physics}, volume 325
  of {\em Grundlehren der Mathematischen Wissenschaften [Fundamental Principles
  of Mathematical Sciences]}.
\newblock Springer-Verlag, Berlin, fourth edition, 2016.

\bibitem{Dafermos1985}
C.~M. Dafermos and W.~J. Hrusa.
\newblock Energy methods for quasilinear hyperbolic initial-boundary value
  problems. {A}pplications to elastodynamics.
\newblock {\em Archive for Rational Mechanics and Analysis}, 87(3):267--292,
  1985.

\bibitem{DLS09}
C.~De~Lellis and L.~Sz{\'e}kelyhidi, Jr.
\newblock The {E}uler equations as a differential inclusion.
\newblock {\em Ann. of Math. (2)}, 170(3):1417--1436, 2009.

\bibitem{DLS13}
C.~De~Lellis and L.~Sz{\'e}kelyhidi, Jr.
\newblock Dissipative continuous {E}uler flows.
\newblock {\em Invent. Math.}, 193(2):377--407, 2013.

\bibitem{DLS14}
C.~De~Lellis and L.~Sz{\'e}kelyhidi, Jr.
\newblock Dissipative {E}uler flows and {O}nsager's conjecture.
\newblock {\em J. Eur. Math. Soc. (JEMS)}, 16(7):1467--1505, 2014.

\bibitem{Demoulini2001}
S.~Demoulini, D.~M.~A. Stuart, and A.~E.~Tzavaras.
\newblock {A} {V}ariational {A}pproximation {S}cheme for {T}hree-{D}imensional
  {E}lastodynamics with {P}olyconvex {E}nergy.
\newblock {\em Archive for Rational Mechanics and Analysis}, 157(4):325--344,
  2001.

\bibitem{Drivas:2017aa}
T.~D. Drivas and G.~L. Eyink.
\newblock An onsager singularity theorem for turbulent solutions of
  compressible euler equations.
\newblock {\em to appear in Communications in Mathematical Physics}, 2017.

\bibitem{DuRo2000}
J.~Duchon and R.~Robert.
\newblock Inertial energy dissipation for weak solutions of incompressible
  {E}uler and {N}avier-{S}tokes equations.
\newblock {\em Nonlinearity}, 13(1):249--255, 2000.

\bibitem{eying}
G.~L. Eyink.
\newblock Energy dissipation without viscosity in ideal hydrodynamics. {I}.
  {F}ourier analysis and local energy transfer.
\newblock {\em Phys. D}, 78(3-4):222--240, 1994.

\bibitem{Feireisletal}
E.~Feireisl, P.~Gwiazda, A.~{\'{S}}wierczewska-Gwiazda, and E.~Wiedemann.
\newblock Regularity and {E}nergy {C}onservation for the {C}ompressible {E}uler
  {E}quations.
\newblock {\em Archive for Rational Mechanics and Analysis}, 223(3):1--21,
  2017.

\bibitem{godunov}
S.~K. Godunov.
\newblock An interesting class of quasi-linear systems.
\newblock {\em Dokl. Akad. Nauk SSSR}, 139:521--523, 1961.

\bibitem{isett}
P.~Isett.
\newblock A {P}roof of {O}nsager's {C}onjecture.
\newblock {\em Arxiv}, (1608.08301), 2016.

\bibitem{Ise13}
P.~Isett.
\newblock {\em H{\"o}lder continuous {E}uler flows in three dimensions with
  compact support in time}, volume 196 of {\em Annals of Mathematics Studies}.
\newblock Princeton University Press, Princeton, NJ, 2017.

\bibitem{KangLee}
E.~Kang and J.~Lee.
\newblock Remarks on the magnetic helicity and energy conservation for ideal
  magneto-hydrodynamics.
\newblock {\em Nonlinearity}, 20(11):2681--2689, 2007.

\bibitem{landau}
L.~D. Landau and E.~M. Lifshitz.
\newblock {\em Electrodynamics of {C}ontinuous {M}edia}.
\newblock Vol. VIII of Course of Theoretical Physics. Pergamon Press, 1961.

\bibitem{LeSh2016}
T.~M. Leslie and R.~Shvydkoy.
\newblock The energy balance relation for weak solutions of the
  density-dependent {N}avier-{S}tokes equations.
\newblock {\em J. Differential Equations}, 261(6):3719--3733, 2016.

\bibitem{shvydkoy_energy}
R.~Shvydkoy.
\newblock On the energy of inviscid singular flows.
\newblock {\em J. Math. Anal. Appl.}, 349(2):583--595, 2009.

\bibitem{shvydkoy}
R.~Shvydkoy.
\newblock Lectures on the {O}nsager conjecture.
\newblock {\em Discrete Contin. Dyn. Syst. Ser. S}, 3(3):473--496, 2010.

\bibitem{Yu}
C.~Yu.
\newblock Energy conservation for the weak solutions of the compressible
  {N}avier--{S}tokes equations.
\newblock {\em Arch. Rational Mech. Anal.}, 225(2):1073--1087, 2017.

\end{thebibliography}

\end{document}